\documentclass[10 pt,reqno]{amsart}
 \usepackage{amsmath,amssymb}
\usepackage{latexsym,color}
\numberwithin{equation}{section}
\newtheorem{theorem}{Theorem}[section]
\newtheorem{corollary}[theorem]{Corollary}
\newtheorem{remark}[theorem]{Remark}

\newtheorem{lemma}[theorem]{Lemma}
\newtheorem{proposition}[theorem]{Proposition}
\newtheorem{example}[theorem]{Example}
\newcommand{\dx}{\mathrm{d}x}

\renewcommand{\Re}{\mathop{\rm Re}\nolimits}
\renewcommand{\Im}{\mathop{\rm Im}\nolimits}

\newcommand{\ta}{\tilde a}
\newcommand{\vk}{\varkappa}

\newcommand{\vp}{\varphi}

\newcommand{\mC}{\mathcal C}
\newcommand{\be}{\begin{equation}}
\newcommand{\ee}{\end{equation}}

\newcommand{\bb}{\mbox{\boldmath$\beta$}}
\newcommand{\R}{{\mathbb R}}
\newcommand{\C}{{\mathbb C}}

\newcommand{\IP}{{\bf P}}
\newcommand{\Z}{{\mathbb Z}}

\newcommand{\E}{{\bf E}}

\newcommand{\T}{{\mathbb T}}

\newcommand{\N}{{\mathbb N}}
\newcommand{\PP}{{\bf P}}
\newcommand{\hA}{\widehat {A}_V}

\newcommand{\ai}{a}
\newcommand{\TTT}{{T}^d}

\newcommand{\fla}{\lan\lan f\ran\ran_\Lambda}

\newcommand{\cA}{{\mathcal A}}

\newcommand{\cD}{{\mathcal D}} 
\newcommand{\cF}{{\mathcal F}}

\newcommand{\cQ}{{\mathcal Q}}

\newcommand{\cP}{{\mathcal P}}
\newcommand{\cY}{{\mathcal Y}}
\newcommand{\strela}{\rightharpoonup}

\def\12{\tfrac12}
\def\S{\mathhexbox278}
\def\lan{\langle}
\def\ran{\rangle}

\def\eps{\varepsilon}

\author[Huang]{ HUANG Guan}
\author[Kuksin]{ KUKSIN Sergei}
\author[Maiocchi]{ MAIOCCHI Alberto}
\title[Time-averaging  for weakly nonlinear CGL equations ]{ Time-averaging  for weakly nonlinear CGL equations with  arbitrary 
potentials
}
\begin{document}

\begin{abstract}
 Consider   weakly nonlinear complex Ginzburg--Landau (CGL) equation of the form:
     \[u_t+i(-\triangle u+V(x)u)=\epsilon\mu\Delta u+\epsilon \mathcal{P}( \nabla u, u),\quad x\in {\R^d}\,,
     \eqno{(*)}\]
under the periodic boundary conditions, 
     where $\mu\geqslant0$ and $\mathcal{P}$ is a smooth function. 
     Let $\{\zeta_1(x),\zeta_2(x),\dots\}$ be the $L_2$-basis formed by  eigenfunctions of the operator $-\triangle +V(x)$. For a complex function $u(x)$,  write it as \mbox{$u(x)=\sum_{k\geqslant1}v_k\zeta_k(x)$} 
     and set $I_k(u)=\frac{1}{2}|v_k|^2$. Then for any solution $u(t,x)$ of the linear equation
      $(*)_{\epsilon=0}$ we have $I(u(t,\cdot))=const$. In this work it is  proved that        if equation $(*)$  with a sufficiently smooth real potential $V(x)$ is well posed on time-intervals $t\lesssim \epsilon^{-1}$, then  for any its solution  $u^{\epsilon}(t,x)$,  the limiting behavior of the curve  $I(u^{\epsilon}(t,\cdot))$ on  time intervals of order $\epsilon^{-1}$, as $\epsilon\to0$, can be uniquely characterized by a solution of a certain well-posed   effective equation:
     \[u_t=\epsilon\mu\triangle u+\epsilon F(u),\]
     where $F(u)$ is a resonant averaging of the nonlinearity $\mathcal{P}(\nabla u, u)$. We also prove similar results for the 
     stochastically perturbed equation, when a white in time and smooth in $x$ random
      force of order $\sqrt\epsilon$ is added to the right-hand side of the equation. 
      
      The approach of this work is rather general. In particular, it applies to equations 
      in bounded domains in $\R^d$ under Dirichlet boundary conditions. 
\end{abstract}

\hfill{\it Dedicated to Walter Craig on his 60-th
birthday } \medskip

\bibliographystyle{plain}
\maketitle

\section{introduction}
\noindent {\it Equations.\/}
We consider a weakly nonlinear CGL equation on a rectangular $d$-torus 
${\TTT}=\mathbb{R}/(L_1\Z)\times\mathbb{R}/(L_2\Z)\times\dots\times\mathbb{R}/(L_d\Z)$,
$L_1,\dots,L_d>0$,
\begin{equation}\label{m-eq}
u_t+i(-\Delta+V(x))u=\epsilon\mu\Delta u+\epsilon\mathcal{P}(\nabla u, u),\quad u=u(t,x),\; x\in{\TTT},\end{equation}
where $\mu\geqslant0$,  $\mathcal{P}: \mathbb{C}^{d+1}\to \mathbb{C}$ is a $C^{\infty}$-smooth function,  $\epsilon$ is a small parameter and $V(\cdot)\in C^n({\TTT})$ is a sufficiently smooth real-valued function on ${\TTT}$ (we will assume that $n$ is large enough). {If $\mu=0$, then the nonlinearity $\mathcal{P}$ should be 
  independent of the derivatives of the unknown function $u$. For simplicity, we assume that  $\mu>0$. The case $\mu=0$ can be treated exactly in the same way (even simpler).}

 For any $s\in\mathbb{R}$ we denote by $H^s$ the Sobolev space of complex-valued  functions
 on $\TTT$, 
 provided with the norm $\|\cdot\|_s$,
 \[\left\|u\right\|_s^2=\langle (-\Delta)^su,u\rangle+\langle u,u\rangle,\quad \text{if}\quad s\ge0\,,\]
 where $\langle\cdot,\cdot\rangle$ is the real scalar product in $L^2({\TTT})$,
 \[\langle u,v\rangle=\Re\int_{{\TTT}}u\bar{v} \dx,\quad u,v\in L^2({\TTT}).\]
 For any $s>d/2+1$, it is  known that the mapping $\mathcal{P}: H^s\to H^{s-1},\;u\mapsto \mathcal{P}(\nabla u,u)$, is smooth and locally Lipschitz, see 
 below Lemma~\ref{lem-smooth}.
 
Our goal is to  study the dynamics of Eq.\ (\ref{m-eq}) on time intervals of
 order $\epsilon^{-1}$ when $0<\epsilon\ll 1$. Introducing the slow time $\tau=\epsilon t$, we rewrite
 the equation  as 
 \begin{equation}\label{m-eq2}
 \dot{u}+\epsilon^{-1}i(-\Delta+V(x))u=\mu\Delta u+\mathcal{P}(\nabla u, u),
 \end{equation}
 where $u=u(\tau,x)$, $x\in{\TTT}$,  and the upper dot  $\dot{}$ stands 
 for $\frac{\mathrm{d}}{\mathrm{d}\tau}$. We assume
 \medskip
 
 \noindent{\bf Assumption A}:  {\it There exists a number 
   $s_*\in(d/2+1, n]$ and for every $M_0>0$ there exists $T=T(s_*, M_0)>0$ such that if $u_0\in H^{s_*}$ and $\|u_0\|_{s_*}\le M_0$, then 
     Eq.~\eqref{m-eq2} has a unique solution $u(\tau,x)\in C([0,T], H^{s_*})$ with the 
     initial datum~$u_0$, and $||u(\tau,x)||_{s_*}\leqslant C(s_*, M_0,T)$ for $\tau\in[0,T]$.}
 \medskip
 
 This assumption can be verified for  Eq.\ \eqref{m-eq} with various nonlinearities
   $\mathcal{P}$.  For
    example   when $\mu=0$ and    $V(x)\equiv0$ it holds if  $\mathcal{P}(u)$ is any smooth function.
    Indeed, taking the scalar product in the space  $H^{s_*}$ of eq.~\eqref{m-eq2} with $u(t)$ and using 
    the Granwall lemma we get the Assumption~A with suitable positive constants $T(s_*, M_0)$ and $C(s_*, M_0,T)$.
   When $\mu>0$, the assumption with any $T>0$ is satisfied by Eq.~\eqref{m-eq} with  nonlinearity  $\mathcal{P}(u)= -\gamma_R f_p(|u|^2)u-i\gamma_If_q(|u|^2)u$, where $\gamma_R,\gamma_I>0$, the functions $f_p(r)$ and $f_q(r)$ are the monomials $|r|^p$ and $|r|^q$, smoothed out near zero, and
$$
0\leqslant p,q<\infty\quad \text{if}\quad d=1,2\quad\text{and}\quad 0\leqslant p,q<\min\left\{\frac{d}{2},\frac{2}{d-2}\right\}\quad \text{if}\quad d\geqslant 3, $$
see, e.g. \cite{hgdcds}.

 We denote by $A_V$ 
 the Schr\"odinger operator  
 $$A_Vu:=-\Delta u+V(x)u.$$
 Let $\{\lambda_k\}_{k\geqslant1}$ be its eigenvalues, ordered in such a way that 
 \[\lambda_1\leqslant\lambda_2\leqslant\lambda_3\leqslant\cdots \,,
 \]
 and let 
$\{\zeta_k,\;k\geqslant1\}$ of $L^2({\TTT})$ be an orthonormal basis, 
formed by the corresponding 
 eigenfunctions. We denote $\Lambda=(\lambda_1,\lambda_2,\dots)$ and call 
 $\Lambda$ the {\it frequency vector} of \mbox{Eq.\ (\ref{m-eq2})}.
  For a complex-valued function $u\in H^s$, we denote by 
 \begin{equation}\label{fourier}
 \Psi(u):=v=(v_1,v_2,\dots),\quad v_k\in\mathbb{C},\end{equation}
  the vector of its Fourier coefficients with respect to the basis $\{\zeta_k\}_{k\geqslant1}$: $u=\sum_{k\geqslant1}v_k\zeta_k$.
  Note that $\Psi$ is a real operator: it maps real functions $u(x)$ to real vectors $v$. 
  In the space of complex sequences $v=(v_1,v_2,\dots)$, we introduce the norms 
  \[\left|v\right|^2_s=\sum_{k=1}^{+\infty}\left(|\lambda_k|^s+1\right)|v_k|^2,\quad
  s\in\mathbb{R}\,,
  \] 
  and denote $h^s=\{v:|v|_s<\infty\}$. Clearly $\Psi$ defines   an isomorphism between the spaces $H^s$ and $h^s$. 
  
  Now we write Eq.\ (\ref{m-eq2}) in the $v$-variables:
  \begin{equation}\label{eq-v1}
  \dot{v}_k+\epsilon^{-1}i\lambda_kv_k=-\mu\lambda_kv_k+P_k(v),\quad k\in\mathbb{N},\end{equation}
  where
   \begin{equation}\label{map-v}
   P(v):=(P_k(v),\;k\in\mathbb{N})=\Psi\Big(\mu V(x)u+\mathcal{P}(\nabla u, u)\Big),\quad u=\Psi^{-1}v.\end{equation}
    For every $k\in\mathbb{N}$ we set 
  \begin{equation}\label{action-angle}
  I_k(v) =\frac{1}{2}v_k\bar{v}_k,\; \text{ and} \;\varphi_k(v)=\mathrm{Arg}\; v_k \in\T^1=\R/(2\pi\Z)
  \;\text{if}\; v_k\neq0,
  \;\text{else}\; \varphi_k=0.
  \end{equation}
   Then $v_k=\sqrt{2I_k}e^{i\varphi_k}$.  Notice that the quantities $I_k$ are conservation laws of the linear equation $(\ref{m-eq})_{\epsilon=0}$, and that the variables 
   $(I,\varphi)\in \R^\infty_+\times \T^\infty$ are its action-angles.
      For any $(I,\vp)\in \R_+^\infty\times \T^\infty$ we denote
   \begin{equation}\label{000}
   v=v(I,\vp)\quad\text{if}\quad v_k = \sqrt{2I_k} e^{i\vp_k},\;\; \forall\, k\,.
   \end{equation}
   If this relation holds, we will write 
   $
   v\sim(I,\varphi)\,.
   $
 We introduce the weighted $l^1$-space $h^s_I$: 
  \[h_I^s:=\{I=(I_k,\;k\in\mathbb{N}   )\in\mathbb{R}^{\infty}:|I|_s^{\sim}=\sum_{k=1}^{+\infty}2(|\lambda_k|^s+1)|I_k|<\infty\}.\]
  Then
  $\ 
  |v|^2_s= |I(v)|^{\sim}_s$,  for each $ v\in h^s
  $.
  Using the action-angle variables $(I,\varphi)$, we  write Eq. (\ref{eq-v1}) as a slow-fast system:
  \[\dot{I}_k=v_k\cdot\big(-\mu\lambda_kv_k+P_k(v)\big),\quad\dot{\varphi}_k=-\epsilon^{-1}\lambda_k+|v_k|^{-2}\cdots,\quad k\in\mathbb{N}.\]
  Here $a\cdot b$ denotes $\Re(a\bar b)$, for $a,b\in\C$, and the dots
  stand for a 
  factor of order 1 (as $\epsilon\to0$).  
  \medskip
  
  \noindent{\it Effective equations.\/}  
  Our  task is to study the evolution of the actions  $I_k$ when $\epsilon\ll 1$ and $0\le\tau\lesssim1$. 
  An efficient way to deal with this problem is through the so-called {\it interaction representation}. Let us define
\begin{equation}\label{def-a}
a_k(\tau)=e^{i\epsilon^{-1}\lambda_k\tau}v_k(\tau)\ .
\end{equation}
  Then 
\begin{equation}\label{newesti}
|a_k|^2=|v_k|^2=2 I_k\ ,
\end{equation}
so to study the evolution of the actions we can use the $a$-variables instead of the $v$-variables. 
Using
  Eq.\ (\ref{eq-v1}), we obtain for $a=(a_1, a_2, \dots)$ the system of equations
  \begin{equation}\label{eq-a1}
  \dot{a}_k(\tau)=-\mu\lambda_ka_k+e^{i\epsilon^{-1}\lambda_k\tau}P_k(\Phi_{-\epsilon^{-1}\Lambda\tau}a),\quad k\in\mathbb{N}\,,
  \end{equation}
  where for each $\theta=(\theta_k,\;k\in\mathbb{N})\in\mathbb{R}^{\infty}$, $\Phi_{\theta}$ stands for 
   the linear operator in $h^s$ defined by 
  \[\Phi_{\theta}v=v',\quad v'_k=e^{i\theta_k}v_k\quad\forall\,k\,.
  \]
  Clearly  $\Phi_\theta$ defines isometries of all   Hilbert spaces $h^s$, and  in the action-angle variables it reads
  $\ 
  \Phi_\theta(I,\vp) = (I,\vp+\theta)\,.
  $
  
  To approximately describe the dynamics of Eq.\ (\ref{eq-a1}) with $\epsilon\ll1$ 
  we introduce an effective equation:
  \begin{equation}\label{eq-ef1}
  \dot{\tilde{a}}_k=-\mu\lambda_k\tilde{a}_k+R_k(\tilde{a}),\quad k\in\mathbb{N},
  \end{equation}
  where $R(\tilde{a}):=(R_k(\tilde{a}),\;k\in\mathbb{N})$ and 
  \begin{equation}\label{limit-ef}R(\tilde{a})=\lim_{T\to\infty}\frac{1}{T}\int_0^{T}\Phi_{\Lambda t}P(\Phi_{-\Lambda t}\tilde{a})dt.\end{equation}
  We will see in Sections \ref{s-av} and  \ref{s-ef} that the limit in \eqref{limit-ef} is well defined and that 
  Eq.~(\ref{eq-ef1}) is well posed, at least locally in time.
  \medskip
  
  \noindent
  {\it Results.\/} 
  In Section~\ref{s-p-m} we prove that the actions of solutions for the effective equation approximate well the actions 
  $I_k(v(\tau))$ of solutions $v$ for \eqref{eq-v1}. Let us fix any $M_0>0$.
    
  \begin{theorem}\label{m-theorem} Let $u(\tau,x)$, $0\le\tau\le T=T(s_*, M_0)$, be a solution
  of \eqref{m-eq2}, such that $u(0,x)=u_0(x)$, 
   $\|u_0\|_{s_*}\le M_0$, existing by Assumption~A. Denote 
  $v(\tau)=\Psi(u(\tau,\cdot))$, $0\le\tau\le T$. Then a solution $\tilde a(\tau)$ of \eqref{eq-ef1},
  such that $\ta(0)=v(0)$,  exists for $0\le\tau\le T$, and for any $s_1< s_*$ we have 
\[ 
\sup_{0\le\tau\le T}
 \left|I(v(\tau) )-I(\tilde{a}(\tau))\right|_{s_1}^{\sim}\to0, \quad \text{as}\quad \epsilon\to0\,.
\]
The rate of the  convergence does not depend on $u_0$,  if $\|u_0\|_{s_*}\le M_0$.
\end{theorem}

This theorem may be regarded as a PDE-version of the Bogolyubov averaging principle, see
\cite{BM} and \cite{AKN}, Section~6.1. The result and its its proof may be easily recasted  to a theorem 
on perturbations of linear Hamiltonian systems with discrete spectrum. Instead of doing this, below we
briefly discuss its generalisations to other nonlinear PDE problems. 
\smallskip

In the second part of the paper (Sections \ref{sec-random}--\ref{sec-r2})
 we consider the CGL equations \eqref{m-eq}
 with added small random force:
\begin{equation}\label{y1}
u_t+i(-\Delta+V(x))u=\epsilon\mu\Delta u+\epsilon\mathcal{P}(\nabla u,u)+\sqrt{\epsilon}\,
\frac{d}{dt} \sum_{l\ge1} b_l \beta_l(t) e_l(x),
\end{equation}
where $u=u(t,x),\; x\in{\TTT}$, the coefficients $b_l$ decay fast enough with $|l|$, $\{\beta_l(t)\}$ are 
standard independent complex Wiener processes and $\{ e_l(x)\}$ is the usual trigonometric basis
of the space $L_2(\TTT)$, parametrized by natural numbers.  It turns out that the effective 
equation for \eqref{y1} is the equation \eqref{eq-ef1}, perturbed by a suitable stochastic forcing,
see Section~\ref{sec-random}. Assuming that the function $\cP$ has at most a polynomial growth 
and that the equation satisfies a suitable stochastic analogy of the Assumption~A we prove
a natural stochastic version of Theorem~\ref{m-theorem} (see Theorem~\ref{f-theorem}). Next,
 supposing that the stochastic
  effective equation is mixing and has a unique stationary measure $\mu_0$,  we prove in 
 Theorem~\ref{t.stat} 
that if $\mu_\epsilon$ is a stationary measure for Eq.~\eqref{y1}, then 
$ \Psi\circ\mu_\epsilon$ converge to $\mu_0$ as $\epsilon\to0$.  So if 
the stochastic  effective equation is mixing, then it 
 comprises asymptotical properties of solutions for Eq.~\eqref{y1}
  as $ t \to\infty$ and $\epsilon\to 0$. 
\medskip

The proof of the theorems in this work follows the Anosov approach to averaging in 
finite-dimensional systems (see in \cite{AKN, LochM}), its version for averaging in resonant systems
(see in \cite{AKN}) and its stochastic version due to Khasminski \cite{Khas68}. The crucial idea that 
for averaging in PDEs the averaged equations for actions (which are equations with 
singularities) should be considered jointly with suitable effective equations (which are regular
 equations)  was suggested in \cite{kuksingafa} for averaging in stochastic PDEs, and later was
 used in \cite{Kuk11} and \cite{hgdcds, hgjdde, skam2013, bplane}.  It was realised in the second group
 of  publications  that for perturbations of linear systems the method may be well combined with the interaction representation of solutions, well known and popular in nonlinear physics (see
 \cite{BM, Naz}), and which  already was used for purposes of completely resonant 
  averaging, corresponding to  constant coefficient PDEs
 with small nonlinearities on the square torus  (see \cite{GG12, FGH}). 
 
 For the case when the spectrum of the unperturbed linear system is non-resonant (see
 below Example~\ref{ex1}), the results of this paper were obtained in \cite{Kuk11, hgdcds}, while
 for the case when the spectrum  is completely resonant -- in
 \cite{skam2013, hgjdde}. The novelty of this work is a version of the Anosov method of averaging, applicable to 
 nonlinear PDEs with small nonlinearities, which does not impose restrictions on the spectrum
 of the unperturbed equation.
 
 Alternatively, the averaging for  weakly nonlinear PDEs may be studied, using the normal 
 form techniques, e.g. see \cite{bambusi1} and references therein. Compared to the Anosov 
 approach, exploited in this work, the method of normal form is much more demanding to the
  spectrum of the unperturbed equation, and more sensitive to its perturbations. So usually it
  applies only in  small vicinities of equilibriums. Its advantage is that it may imply stability on
  longer time intervals, while the method of this work is restricted to the first-order averaging. So
  in the deterministic setting it allows to control solutions of $\epsilon$-perturbed equations 
  only on time-intervals of order $\epsilon^{-1}$ (still, in the stochastic setting it also allows to 
  control the stationary measure, which describes the asymptotic behaviour of solutions as
  $t\to\infty$).
  \medskip
  
    \noindent
  {\it Generalizations.\/}  The Anosov-like method of resonant averaging, presented in this work,
   is very flexible. With some slight changes,
   it easily generalizes to weakly nonlinear CGL equations, involving
   high order derivatives,
\begin{equation}\label{dcdse}
u_t+i(-\triangle u+V(x)u)=\epsilon \mathcal{P}(\nabla^2 u,\nabla u, u,x),\quad x\in {\TTT},\end{equation}
 provided  that the Assumption~A holds and the 
corresponding effective equation is well posed locally in time. See in Appendix A
(also see  \cite{hgdcds}, where a similar result is 
proven for the case of   non-resonant spectra). 

The method applies to equations \eqref{m-eq}  and \eqref{y1} in a bounded domain
$\mathcal O\subset\R^d$ under Dirichlet boundary conditions. Indeed,
if $d\le3$, then to treat the corresponding boundary-value problem 
 we can literally repeat the argument of this work, replacing there the space $H^s$
with the Hilbert space 
$H^2_0({\mathcal O})=\{u\in H^2 ({\mathcal O}) :  u\mid_{\partial {\mathcal O}}=0\}$.
If $d\ge4$, then $H^s$ should be replaced with an  $L_p$-based  Banach space 
$W^{2,p}_0({\mathcal O})$, where $p>d/2$. 

Obviously the method applies to weakly nonlinear equations of other types; e.g. to weakly
 nonlinear wave equations. In \cite{bplane} the method in its stochastic form 
  was applied to the Hasegawa-Mima equation,
 regarded as a perturbation of the Rossby equation $(-\Delta +K)\psi_t (t,x,y)- \psi_x=0$,
 while in \cite{Dym} it is applied to systems of non-equilibrium statistical physics, where each particle 
 is perturbed by an $\eps$-small Langevin thermostat, and  is studied 
 the limit $\eps\to0$   (similar to the same limit in Eq. \eqref{y1}).

The  averaging for perturbations of nonlinear integrable PDEs is more complicated.
Due to the lack in  the functional phase-spaces of an analogy  of the Lebesgue measure (required 
  by the Anosov approach to the finite-dimensional deterministic averaging),  in this case the results
   for stochastic perturbations are significantly 
stronger than the deterministic results. See in \cite{HGK}.

 \medskip
 \noindent{\bf Acknowledgments.}  We are thankful to Anatoli Neishtadt for discussing the 
 finite-dimensional averaging.
 This work was supported by  l'Agence
Nationale de la Recherche through the grant STOSYMAP (ANR 2011BS0101501).

\section{Resonant averaging in Hilbert Spaces}\label{s-av}
The goal of this section is to show that the limit  in   (\ref{limit-ef}) is well-defined in some suitable settings and study its properties.
 Below for an infinite-vector $v=(v_1,v_2,\dots)$ and any $m\in\N$ we denote
$$
v^m=(v_1,\dots,v_m), \quad\text{or}\quad v^m=(v_1,\dots,v_m, 0,\dots),
$$
depending on the context. This agreement 
 also applies to elements $\vp=(\vp_1,\vp_2,\dots)$ of the torus $\T^\infty$. 
For $m$-vectors $I^m, \vp^m, v^m$ we write $v^m \sim (I^m, \vp^m)$ if \eqref{000}
 holds for $k=1,\dots,m$. 
By $\Pi^m$, $m\geqslant1$, we denote the  Galerkin  projection
\[\Pi^m: h^0\to h^0, (v_1,v_2,\dots)\mapsto v^m=(v_1,\dots,v_m,0,\dots).\]

For a continuous complex
 function $f$ on a Hilbert space $H$, we say  that  $f$ is locally Lipschitz and write 
 $f\in Lip_{loc}(H)$ if 
\begin{equation}\label{lip-c1}\left|f(v)-f(v')\right|\leqslant \mathcal{C}(R)\|v-v'\|,\quad \text{if}\quad \|v\|,\|v'\|\leqslant R,\end{equation}
for some  continuous non-decreasing 
 function $\mathcal{C}:\R^+\to  \R^+$ which depends on $f$. We  write
 \be\label{lip-c2}
 f\in Lip_{\mathcal C}(H) \;\; \text{if \eqref{lip-c1} holds and  $|f(v)| \le {\mathcal C}(R)$ if }\ \|v\|\le R\,.
 \ee
 If $ f\in Lip_{\mathcal C}(H)$, where ${\mathcal C}(\cdot)=\,$Const, then $f$ is a bounded
 (globally)  Lipschitz  function.
 If $B$ is a Banach space, then the space $Lip_{loc}(H,B)$ of locally Lipschitz mappings $H\to B$ and its subsets 
 $Lip_{\mathcal C}(H,B)$ are defined similarly. 
 \smallskip
 
 For any vector 
$W=(w_1,w_2,\dots)\in\mathbb{R}^{\infty}$ we set 
 \begin{equation}
\langle f\rangle^T_{W,l}(v)=\frac{1}{T}\int_0^Te^{iw_lt}f(\Phi_{-Wt}v)dt,
\end{equation}
and if the limit of $\langle f\rangle_{W,l}^{T}(v)$ when $T\to\infty$ 
exists, we denote 
$$\langle f\rangle_{W,l}=\lim_{T\to\infty}\langle f\rangle_{W,l}^T(v).$$
Concerning this definition we have the following lemma. Denote
$$
B(M,h^s)=\{v\in h^s: |v|_s\leqslant M\},\quad M>0.
$$
\begin{lemma}\label{lem-def-av} Let $f\in Lip_{\mC}(h^{s_0})$ for some $s_0\geqslant0$ and some function $\mC$ 
as above.Then 

(i)  For every $T\ne0$, $\langle f\rangle_{W,l}^T \in Lip_{\mC}(h^{s_0})$.

(ii) 
 The limit 
$\langle f\rangle_{W,l}(v)$ exists for $v\in h^{s_0} $ and this function 
 also belongs to $Lip_{\mC}(h^{s_0})$.

(iii) For $s>s_0$ and 
 any $M>0$,  the functions 
  $\langle f\rangle_{W,l}^T(v)$ converge, as  $T\to\infty$, to 
$\langle f\rangle_{W,l}(v)$ uniformly for $v\in B(M,h^s)$. 

(iv) The convergence is uniform for $f\in Lip_{\mC}(h^{s_0})$ with a fixed 
function ${\mC}$. 
\end{lemma}
\begin{proof} (i) It is obvious since the transformations $\Phi_\theta$ are isometries of $h^{s_0}$.

(ii)  To prove this, consider the 
restriction of $f$ to $B(M,h^{s_0})$, for any fixed $M>0$. Let us take some $v \in B(M,h^{s_0})$
and fix any $\rho>0$. Below in this proof by $O(v)$, $O_1(v)$, etc, we denote various functions 
$g(v)=g(I, \vp)$, defined for 
 $|v |_{s_0}\le M $  and bounded by $1$. 

Let us choose any { $m=m(\rho,M, v, \mC)$ }such that 
$$
\mathcal{C}(M)\, |v-\Pi^m v|_{s_0} \le \rho\,.
$$
Then 
$
|f(v)-f(\Pi^mv)|<\rho,
$
and by (i) 
$$\big|\langle f\rangle_{W,l}^{T}(v)-\langle f\rangle_{W,l}^T(\Pi^mv)\big|<\rho\,,
$$
for every $T>0$. 

Let us  set
 $$
  \mathcal{F}^m(I^m,\vp^m) =
  \mathcal{F}^m(v^m)=   f \big(v^m\big),\qquad
   \forall\, v^m\sim(I^m, \vp^m) \in \C^m\,,
   $$
   where in the r.h.s.  $v^m$ is regarded as the vector $(v^m,0,\dots)$.
 Clearly, the function  $\vp^m\mapsto  \mathcal{F}^m(I^m,\varphi^m)$ is  Lipschitz-continuous 
  on $\mathbb{T}^m$. So    its Fejer polynomials
   $$
   \sigma_K(\cF^m)=
   \sum_{k\in\mathbb{Z}^m, \, |k|_\infty\le K}a_k^K e^{ik\cdot(\varphi^m)}, \qquad K\ge1\,,
   $$
  where ${a_k^K}={a_k^K}(m,I^m)$, 
   converges to 
$ \mathcal{F}^m(I^m,\varphi^m)$ uniformly on $\mathbb{T}^m$.
 Moreover, the rate of convergence depends only on its Lipschitzian norm and  the dimension $m$ (see e.g.\ Theorem 1.20, Chapter \uppercase\expandafter{\romannumeral17}  of \cite{zygmund}).  Therefore, there exists $K=K(\mC, M,\rho, m)>0$ such that 
\be\label{decomp}
 \mathcal{F}^m(I^m,\varphi^m)=
 \sum_{k\in\Z^m,   {|k|_\infty}\leqslant K}{a_k^K}e^{ik\cdot\varphi^m} + \rho  O_1  (I^m,\vp^m)\,.
\ee
Now we define 
$$
\mathcal{F}^{res}_K(I^m,\varphi^m)=\sum_{k\in S(K)}{a_k^K}e^{ik\cdot\varphi^m},\quad S(K)=
\{k\in\mathbb{Z}^m: \; {|k|_\infty}\leqslant K, w_l-\sum_{j=1}^mk_iw_i=0\}.
$$

Since 
$$ 
 \mathcal{F}^m
\big(\Phi_{-W^m t}(\Pi^mv)\big) = \mathcal{F}^m(I^m,\varphi^m-Wt)\,,
$$
then
\begin{equation*}
\begin{split}
\langle e^{ik\cdot \vp^m}  \rangle_{W,l}^T  &= e^{ik\cdot \vp^m} \quad\text{if} \quad k\in S(K)\,,\\
\Big| \langle e^{ik\cdot \vp^m}  \rangle_{W,l}^T \Big| &\le \frac{2T^{-1}}{|w_l-k\cdot W^m|} \quad\text{if} \quad {|k|_\infty}\le K,\;
 k\notin S( K)\,,
\end{split}
\end{equation*}
where we regard $e^{ik\cdot \vp^m}$ as a function of $v$.  Accordingly,
\begin{equation*}
\begin{split}
\langle f\rangle_{W,l}^T (v)& = \langle \mathcal{F}^m(I^m,\varphi^m)
\rangle_{W,l}^T + \rho O_2 (v)\\
&=
\mathcal{F}^{res}_K(I^m,\varphi^m) + {C(\rho, M,W,f,I)}{T^{-1}} O_3(v) +\rho O_4(v)
\,.
\end{split}
\end{equation*}
 So there exists {$\bar T=T(\rho, M, W, f,I)>0$} such that if $T\geqslant \bar T$, then 
$$
\left|\langle f\rangle_{W,l}^T-\mathcal{F}_K^{res}(I^m,\varphi^m)\right|<2\rho\,,
$$
and  for any $T'\geqslant T''\geqslant \bar T$, we have 
$$\big|\langle f\rangle_{W,l}^{T'}(v)-\langle f\rangle_{W,l}^{T''}(v)\big|<4\rho.
$$
This implies that  the limit 
$\langle f\rangle_{W,l}(v)$ exists for every $v\in B(M,h^{s_0})$.  Using (i) we obtain that 
 $\langle f\rangle_{W,l}(\cdot)\in Lip_{\mC}(h^{s_0})$. 

(iii) This statement  follows  directly from (ii) since the family of  
 functions $\{\langle f\rangle_{W,l}^{T}(v)\}$ is uniformly continuous on balls $B(R,h^{s_0})$ by (i) 
and each ball $B(M,h^s)$,  $s>s_0$, is compact in $h^{s_0}$.

(iv) From the proof of (ii) we see that for any $\rho>0$ and $v\in h^{s_0}$, there exists $T=T(W,\rho,v,\mC)$ such that if $T'\geqslant T$, then $|\langle f\rangle_{W,l}^{T'}(v)-\langle f\rangle_{W,l}(v)|\leqslant \rho$. This implies the assertion.
\end{proof}
We now give some examples of the limits $\langle f\rangle_{W,l}$.

\begin{example}\label{ex1}
If the vector $W$ is non-resonant, i.e., non-trivial finite linear combinations of  $w_j$'s 
with integer coefficients do not vanish (this property holds for typical potentials $V(x)$, see \cite{Kuk11}), then 
the set $S(K)$ reduces to one trivial resonance $e_l=(0,\dots,0,1,0,\dots0)$, where 1 stands on the $l$-th place 
(if $m< l$, then  $S(K)=\emptyset$). Let $f(v)$ be any finite polynomial of $v$. We write it in the form 
$\sum_{k,l \in\N^\infty, |k|,|l|<\infty} f_{k,l}(I)v^k \bar v^l$, where $f_{k,l}$ are polynomials of $I$ and finite vectors $k,l$
are such that if  $k_j\ne0$, then $l_j=0$, and vice versa. Then 
$\langle f\rangle_{W,l} =  f_{e_l,I}(I) v_l$.
\end{example}

\begin{example}\label{ex2}
If $f$ is a  linear functional, 
  $f=\sum_{i=1}^{\infty} b_i v_i$, then for any $ l\in\mathbb{N}$, 
\[\langle f\rangle_{W,l}=\sum_{i\in \mathcal{A}^1_l}b_iv_i,\quad \mathcal{A}^1_l=\{i\in \mathbb{N}: w_i-w_l=0\}.\]
 If $f$ is polynomial of $v$, e.g. $f=\sum_{i+j+m=k}a_{i,j,k}v_iv_jv_k$, then 
\[\langle f\rangle_{W,l}=\sum_{(i,j,m)\in \mathcal{A}_l^3}a_{i,j,k}v_iv_jv_k,\quad 
\mathcal{A}_l^3=\{(i,j,m)\in\mathbb{N}^3: w_l-w_i-w_j-w_m=0\}.\]
\end{example}

We may also consider the averaging 
\be\label{newav}
\begin{split}
 \lan\lan f\ran\ran^T_W(v)=\frac1{T} \int_0^T f(\Phi_{-Wt}v)\,dt\,,\quad
\lan\lan f\ran\ran _W(v) =\lim_{T\to\infty}  \lan\lan f\ran\ran^T_W(v)\,.
\end{split}
\ee

\begin{lemma}\label{lnew}
Let $f\in Lip_{\mC}(h^{s_0})$. Then 

a) for the averaging $\lan\lan \cdot \ran\ran _W$ hold natural analogies 
of all assertions of Lemma~\ref{lem-def-av}. 

b) The function $\lan\lan f\ran\ran _W$ commutes with the transformations $\Phi_{Wt}$, $t\in\R$.
\end{lemma}
\begin{proof}
To prove a) we repeat for the averaging $\lan\lan \cdot \ran\ran _W$ the proof of Lemma \ref{lem-def-av}, replacing there $w_l$ by 0.
Assertion b) immediately follows from the formula for  $\lan\lan f\ran\ran _W^T$ in \eqref{newav}.
\end{proof}

\section{The effective equation}\label{s-ef}
Let $V(x)\in C^n({\TTT})$. As in the introduction, $A_V$ is the operator
$-\Delta+V$ and $\{\lambda_k,k\in\mathbb{N}\}$ 
are its  eigenvalues.

The following result is well known, see  Section~5.5.3 in \cite{sobo1996}.
\begin{lemma}
If $f(x):\mathbb{C}\to \mathbb{C}$ is $C^{\infty}$, then the mapping 
\[M_f: H^s\to H^s,\quad u\mapsto f(u),\]
is $C^{\infty}$-smooth for $s>d/2$. Moreover, $M_f\in Lip_{\mC_s}(H^s,H^s)$
for a suitable function~$\mC_s$.
\label{lem-smooth}
\end{lemma}

 Consider the map $P(v)$ defined in (\ref{map-v}). From Lemma~\ref{lem-smooth}, we have 
\begin{equation}\label{lip-p-v}
P(\cdot)\in Lip_{{\mathcal C}_s}(h^s,h^{s-1}),\quad  \forall\,s\in(d/2+1, n]\,,
\end{equation}
for some ${\mathcal C}_s$. 
We recall that $\Lambda$ is the frequency vector of Eq. (\ref{m-eq2}). 
For any $T\in\mathbb{R}$, we denote $$\langle P\rangle_{\Lambda}^T(v):=(\langle P_k\rangle_{\Lambda,k}^T(v),k\in\mathbb{N})=\frac{1}{T}\int_0^{T}\Phi_{\Lambda t}P(\Phi_{-\Lambda t}v)dt\,,
$$
and $$R(v)=\langle P\rangle_{\Lambda}(v):=(\langle P_k\rangle_{\Lambda,k}(v),k\in \mathbb{N})\,.
$$
\begin{example}
If $P$ is a diagonal operator, $P_k(v)=\gamma_k v_k$ for each $k$, where $\gamma_k$'s are complex 
numbers, then in view of Example~\ref{ex2}, $\langle P\rangle_{\Lambda} = P$.
\end{example}

We have the following lemma:
\begin{lemma}\label{lem-r-lip} (i)
For every $d/2<s_1<s-1\leqslant n-1$ and $M>0$, we have 
\begin{equation}\label{conv}
\Big|\langle P\rangle_{\Lambda}^T(v)-R(v)\Big|_{s_1}\to0,\quad\text{as}\quad T\to\infty,
\end{equation}
uniformly for $v\in B(M,h^{s})$;

(ii) $R(\cdot)\in Lip_{{\mathcal C}_s}(h^s,h^{s-1})$, $s\in(d/2+1,n]$;

(iii) $R$ commutes with $\Phi_{\Lambda t}$, for each $t\in\R$. 
\end{lemma}
\begin{proof} (i) There exists $M_1>0$, independent from $v$ and $T$, such that 
$$\left|\langle P\rangle_{\Lambda}^T(v)-R(v)\right|_{s-1}\leqslant M_1,\quad v\in B(M,h^{s}).$$
So for any $\rho>0$ we can find  $m_{\rho}>0$ such that 
$$\Big|(\mathrm{Id}-\Pi^{m_\rho})\big[\langle P\rangle_{\Lambda}^T(v)-R(v)\big]\Big|_{s_1}<\rho/2,\quad v\in B(M,h^{s}).$$
By Lemma~\ref{lem-def-av}(iii), there exists $T_{\rho}$ such that for $T>T_{\rho}$, 
\[ \Big|\Pi^{m_{\rho}}\big[\langle P\rangle_{\Lambda}^T(v)-R(v)\big]\Big|_{s_1}<\rho/2,\quad v\in B(M, h^{s}).\]
Therefore if $T>T_{\rho}$, then 
\[\left|\langle P\rangle_{\Lambda}^T(v)- R(v)\right|_{s_1}<\rho,\quad v\in B(M,h^{s}).\]
This implies the first assertion.

(ii) 
Using the fact that the linear maps $\Phi_{\Lambda t}$,
$t\in\mathbb{R}$ are isometries in $h^{{s}}$, we obtain that for
$T\in\mathbb{R}$ and  $v',v''\in B(M,h^{{s}})$, 
$$\left|\langle P\rangle_{\Lambda}^T(v')-\langle
P\rangle_{\Lambda}^T(v'')\right|_{{s-1}}\leqslant
\mC_s(M)\left|v'-v''\right|_{{s}}\ .$$ 
Therefore 
$$\left|R(v')-R(v'')\right|_{{s-1}}\leqslant \mC_s(M)\left|v'-v''\right|_{{s}}, \quad v',v''\in B(M,h^{{s}}).$$
This estimate, the convergence \eqref{conv} and the Fatou lemma imply that $R$ is a locally Lipschitz 
mapping with a required estimate for the Lipschitz constant. A bound on its norm may be obtained in
a similar way, so the second assertion follows.

(iii) We easily verify that 
$$
\big|
\langle P\rangle_{\Lambda}^{T+{t}}(v) -\Phi_{\Lambda{t}} \langle P 
\rangle_{\Lambda}^T(  \Phi_{-\Lambda{t}}
 v)\big|_{s-1} \le 2 \,\mC_s(|v|_s)\frac{|{t}|}{|T+{t}|}\,.
$$
Passing to the limit as $T\to\infty$ we recover (iii). 
\end{proof}

\begin{corollary}\label{c1}
For $d/2<s_1<s-1\le n-1$ and any $v\in h^s$,
$$
\langle P\rangle^T_\Lambda(v) = R(v) +\vk(T;v),
$$
where $|\vk(T;v)|_{s_1} \le \overline\vk (T;|v|_s)$.  Here for each $T$, $ \overline\vk (T;r)$ is an 
increasing function of $r$, and for each $r\ge0$, $ \overline\vk (T;r)\to0$ as $T\to\infty$. 
\end{corollary}

\begin{example}In the completely resonant case, when 
\be\label{complres}
 L_1=\dots =L_d=2\pi\quad \text{and}\quad V=0\,,
 \ee
 the frequency vector is $\Lambda=(|\mathbf{k}|^2,\mathbf{k}\in\mathbb{Z}^d)$.
 If
  $\mathcal{P}(u)=i|u|^2u$, then
\[
P(v)=(P_{\mathbf{k}}(v),\mathbf{k}\in\mathbb{Z}^d), \quad v=(v_{\mathbf{k}},\mathbf{k}\in\mathbb{Z}^d), \quad u=\sum_{\mathbf{k}\in\mathbb{Z}^d}v_{\mathbf{k}}e^{i\mathbf{k}\cdot x},
\]
with
\[P_{\mathbf{k}}(v)=\sum_{\mathbf{k}_1-\mathbf{k}_2+\mathbf{k}_3=\mathbf{k}}iv_{\mathbf{k}_1}\bar{v}_{\mathbf{k}_2}v_{\mathbf{k}_3},\quad \mathbf{k}\in\mathbb{Z}^d.\]
 Therefore 
$\langle P\rangle_{\Lambda}=(\langle P_{\mathbf{k}}\rangle_{\Lambda,\mathbf{k}},\mathbf{k}\in\mathbb{Z}^d)$, with
\[\langle P_{\mathbf{k}}\rangle_{\Lambda,\mathbf{k}}=\sum_{(\mathbf{k}_1,\mathbf{k}_2,\mathbf{k}_3)\in Res(\mathbf{k})}iv_{\mathbf{k}_1}\bar{v}_{\mathbf{k}_2}v_{\mathbf{k}},
\]
where
$\  Res(\mathbf{k})=\{(\mathbf{k}_1,\mathbf{k}_2,\mathbf{k}_3): |\mathbf{k}_1|^2-|\mathbf{k}_2|^2+|\mathbf{k}_3|^2-|\mathbf{k}|^2=0\}$. 
\end{example}

 Lemma~\ref{lem-r-lip} implies that the effective equation (\ref{eq-ef1}) is a
quasi-linear heat equation. So it is locally well-posed in the   spaces $h^s$, $s\in (d/2+1,n]$.


\section{Proof of the averaging  theorem}\label{s-p-m}
In this section we will prove   Theorem \ref{m-theorem}.  We recall
that $d/2+1<s_*\le n$ and $s_1<s_*$, where  $s_*$  
is the number from Assumption~A and $n$ is a sufficiently big integer (the smoothness of the potential $V(x)$). 
 Without loss of
generality we assume that 
$$
s_1>d/2+1\quad \text{and}\quad 
s_1> s_*-2\,,
$$ and that  Assumption~A holds with $T=1$.

 Let 
 $u^\epsilon(\tau,x)$ be the solution of Eq. (\ref{m-eq2}) from Theorem~\ref{m-theorem},
 $$
 \|u^\epsilon(0,x)\|_{s_*}\leqslant M_0\,,
 $$
 and 
 $\ 
 v^\epsilon(\tau)=\Psi(u^\epsilon(\tau,\cdot)).
 $ 
 Then there exists $M_1\geqslant M_0$ such that 
\[v^\epsilon(\tau)\in B( M_1, h^{s_*}),\quad\tau\in[0,1]\,,
\]
for each $\epsilon>0$. 
The constants in estimates below in this section may depend on $M_1$, and this dependence
may  be non-indicated.

Let 
$$
a^\epsilon(\tau)=\Phi_{\tau\epsilon^{-1}\Lambda }(v^\epsilon(\tau))
$$ 
be  the interaction representation of~$v^\epsilon(\tau)$ (see Introduction), 
$$
a^\epsilon(0) = v(0)=:v_0\,.
$$
For every \mbox{$v=(v_{k},k\in\mathbb{N})$}, denote 
$$
\hA
(v)=(\lambda_{k}v_{k},k\in\mathbb{N})=\Psi(A_Vu)\ ,\quad
u=\Psi^{-1}v\ .
$$ 
 Then 
\be\label{4.0}
\dot{a}^\epsilon(\tau)=
-\mu{\hA}(a^\epsilon(\tau))+Y\big(a^\epsilon(\tau),\epsilon^{-1}\tau\big)\,,
\ee
where 
\begin{equation}\label{def-Y}
  Y\big(a, t \big)=
 \Phi_{t \Lambda}\Big(P\big(\Phi_{-t \Lambda}(a)\big)\Big)\,.
\end{equation}
Let $r\in(d/2+1,n]$. 
Since  the operators $\Phi_{t\Lambda}$, $t\in\mathbb{R}$, define isometries of $h^{r}$, then, in view of 
 (\ref{lip-p-v}),  for any $t\in\mathbb{R}$ we have 
\begin{equation}
Y(\cdot, t)\in Lip_{\mC_r} (H^r, H^{r-1})\,.
\label{rnls-r-lip2}
\end{equation}

For any $s\ge0$ we denote by $X^s$ the space
$$
X^s = C([0,T], h^s)\,,
$$
given the supremum-norm. Then
\be\label{new1}
|a^\epsilon|_{X^{s_*}}\le M_1,\qquad   |\dot a^\epsilon|_{X^{s_*-2}}\le  C(M_1)\,.
\ee
Since for $0\le \gamma\le1$ we have
$$
|v|_{\gamma(s_*-2)+(1-\gamma)s_*} \le |v|^\gamma_{s_*-2} |v|^{1-\gamma}_{s_*}
$$
by the interpolation inequality, then in view of \eqref{new1} for any $s_*-2<\bar s<s_*$ and 
$0\le\tau_1\le\tau_2\le1$ we have 
\be\label{new2}
| a^\epsilon(\tau_2)-a^\epsilon(\tau_1)|_{\bar s} \le C(M_1)^\gamma (\tau_2-\tau_1)^\gamma
(2M_1)^{1-\gamma}\,,
\ee
for a suitable $\gamma=\gamma(\bar s, s_*)>0$, uniformly in $\epsilon$.

Denote 
$$
\mathcal{Y}(v,t)=Y(v,t)-R(v).
$$
 Then by Lemma~\ref{lem-r-lip} 
  relation (\ref{rnls-r-lip2}) also holds for the map
  $v\mapsto\mathcal{Y}(v,t)$, for any $t$. 

The following lemma is the main step of the proof.
\begin{lemma} \label{rnls-lem-apr}
 For every $s'>d/2+1$, $s_*-2<s'<s_*$ we have
\be\label{final}
\Big |\int_{0}^{\tilde{\tau}}\mathcal{Y}(a^\epsilon(\tau),\epsilon^{-1}\tau)d\tau\Big|_{s'}\leqslant \delta(\epsilon, M_1), \quad \forall\,
 \tilde{\tau}\in[0,1],
\ee
where $\delta(\epsilon, M_1)\to0$ as $\epsilon\to0$.
\end{lemma}
\begin{proof}   Below in this proof we write $a^\epsilon(\tau)$ as $a(\tau)$. 
We divide the time interval $[0,1]$ into subintervals~$[b_{l-1}, b_{l}]$, $l=1,\cdots, N$ of length 
$L=\epsilon^{1/2}$:
\[
b_k=Lk\quad   \text{for}\quad k=0,\dots, N-1,\;\; b_N=1  \,,
b_N-b_{N-1}\leqslant L\ ,
\]
where $N\leqslant 1/L+1\le 2/L$.

In virtue of \eqref{rnls-r-lip2} and Lemma~\ref{lem-r-lip} (ii),
\be\label{new}
\Big| \int_{b_{N-1}}^{b_{N}  }\mathcal{Y}(a(\tau) ,\epsilon^{-1}
\tau)\,d\tau\Big|_{s'}\le L C(s',s,M_1)\ .
\ee
Similar, if $\bar\tau\in[b_r, b_{r+1})$ for some $0\le r<N$, then
$|\int_{b_r}^{\bar\tau} \mathcal{Y}\,d\tau|_{s'}$ is bounded by the r.h.s. of \eqref{new}.

Now we estimate the integral of $\mathcal {Y}$ over any segment $[b_l, b_{l+1}]$, where 
 $l\le N-2$. To do this we write  it  as 
\[ \begin{split}
\int_{b_{l}}^{b_{l+1}  }\mathcal {Y}(a(\tau) ,\epsilon^{-1} \tau)\,d\tau &=
\int_{b_l}^{b_{l+1} } \big({Y}(a(b_l),\epsilon^{-1}\tau) -  R(a(b_l))\big)\,d\tau
\\
&+
\int_{b_{l}}^{b_{l+1}  }\big(  {Y}(a(\tau) ,\epsilon^{-1} \tau)  - {Y}(a(b_l) ,\epsilon^{-1} \tau)\big)\,d\tau\\
&+\int_{b_l}^{b_{l+1}}\big(R(a(b_l))-R(a(\tau))\big)\, d\tau\,.
\end{split}
\]
In view of Lemma \ref{lem-r-lip} and \eqref{new2}  the $h^{s'}$-norm of the second and  third terms  in the r.h.s. are bounded by 
$
C(s', s, M_1) L ^{1+\gamma}.
$
Since
\[ \begin{split}
\epsilon \int_0^{{\epsilon^{-1}L}} Y(a(b_l), \epsilon^{-1} b_l+s)\,ds 
=
L\Phi_{\Lambda \epsilon^{-1} b_l}
\frac1{{L^{-1}}}\int_0^{{L^{-1}}} \Phi_{\Lambda s}
P(\Phi_{-\Lambda s} 
(\Phi_{-\Lambda\epsilon^{-1} b_l} a(b_l))\big)ds\, ,
\end{split}
\]
then 
using Corollary \ref{c1}  and Lemma \ref{lem-r-lip} (iii) we see that this equals 
$
L R(a(b_l)) +\vk_1({L^{-1}}),
$
where $| \vk_1({L^{-1}})|_{s'} \le
\overline\vk({L^{-1}};M_1)$ and $ \overline\vk \to0$ when
$L^{-1}\to\infty$.  
We have arrived at the estimate 
\be\label{foralberto}
\Big| \int_{b_{l}}^{b_{l+1}  }\mathcal{Y}(a(\tau) ,\epsilon^{-1} \tau)\,d\tau\Big|_{s'}\le L\Big(
\overline\vk({L^{-1}};M_1) + C L^{\gamma}\Big)\,.
\ee
Since $N\le 2/L$ and $L = \epsilon^{1/2}$, then by \eqref{foralberto} and \eqref{new} 
the l.h.s. of 
\eqref{final}  is bounded by
$
2\overline\vk({\epsilon^{-1/2}};M_1)  +C\epsilon^{\gamma/2}+C\epsilon^{1/2}.
$
It implies the assertion
 of the lemma.
\end{proof}

Consider the effective equation \eqref{eq-ef1}. By Lemma~\ref{lem-r-lip} this is the linear parabolic equation
$\dot u-\Delta u+V(x)u=0$, written in the $v$-variables, perturbed by a locally Lipschitz operator of order one. 
So its solution $\ta(\tau)$ such that $\ta(0)=v_0$ exists (at least) locally in time.  Denote 
by $\tilde T$ the stopping time 
$$\tilde{T}=\min\{\tau \in[0,1] : \; |\tilde{a}(\tau)|_{s_*}\geqslant M_1+1\}\,,
$$
where, by definition,  $\min\emptyset=1$.

Now consider the family of curves $a^\epsilon(\cdot)\in  X^{s_*}$. In view of \eqref{new1}, \eqref{new2} and the 
 Arzel\`a-Ascoli theorem (e.g. see in \cite{KelNam}) 
  this family is pre-compact in each space $X^{s_1}$, $s_1<s_*$. Hence, for any sequence 
 $\epsilon_j'\to0$ there exists a subsequence $\epsilon_j\to0$  such that 
$$
a^{\epsilon_j}(\cdot)\underset{\epsilon_j\to0}\longrightarrow a^0(\cdot)\quad \text{in}\quad X^{s_1}\,.
$$
By this convergence, \eqref{new1} and the Fatou lemma, 
\be\label{new5}
|a^0(\tau)|_{s^*} \le M_1\qquad \forall\, 0\le\tau\le1\,.
\ee
In view of  Lemma \ref{rnls-lem-apr}, the curve $a^0(\tau)$ is a mild solution of Eq.~\eqref{eq-ef1} in the space $h^{s_1}$,
i.e,
$$
a(\tau) - a(0) = \int_0^\tau \big(-\mu \hA a(s) +R(a(s))\,ds\,,\quad
\forall\, 0\le\tau\le1
$$
(the equality holds in the space $h^{s_1-2}$). So $a^0(\tau)= \ta(\tau)$ for $0\le\tau\le\tilde T$. In view of \eqref{new5} and the
definition of the stopping time $\tilde T$ we see that $\tilde T=1$. That is, $\tilde a\in X^{s_*}$ and  
\begin{equation}\label{convergence1}
a^{\epsilon}(\cdot)
\longrightarrow
\tilde{a}(\cdot)\quad \text{in} \quad X^{s_1} \,,
\end{equation}
where $\epsilon=\epsilon_j\to0$. Since the limit $\tilde a$ does not depend on the sequence $\epsilon_j\to0$, then 
the convergence holds as $\epsilon\to0$.

Now we  show that the convergence \eqref{convergence1} holds uniformly for 
 $v_0\in B(M_0, h^{s_*})$. Assume the opposite. Then there exists $\delta>0$, sequences $\tau_j\in[0,1], a_0^j\in B(M_0,h^{s_*})$, and $\epsilon_j\to0$ such that if $ a^{\epsilon_j}(\cdot) $ is a solution of \eqref{4.0} 
  with initial data $a_0^j$ and $\epsilon=\epsilon_j$, and $\tilde{a}^{j}(\cdot)$ is a
  solution of the effective equation \eqref{eq-ef1} with the same initial data, then 
\begin{equation}\label{contradict}
|a^{\epsilon_j}(\tau_j)-\tilde{a}^j(\tau_j)|_{s_1}\geqslant\delta.
\end{equation}

Using again  the Arzel\`a-Ascoli theorem and \eqref{new2},  replacing the subsequence $\epsilon_j\to0$ by a suitable
subsequence,  we have that \begin{align}
&\tau_j\to\tau_0\in [0,1],\notag\\
&a^j_0\to a_0\quad \text{in}\quad h^{s_1}\,, \quad \text{where}\quad a_0\in h^{s_*},\notag\\
&a^{\epsilon_j}(\cdot) \to a^0(\cdot)\quad \text{in} \quad  X^{s_1},\notag\\
&\tilde{a}^j(\cdot)\to \tilde{a}^0(\cdot)\quad \text{in}\quad X^{s_1}.\notag
\end{align}
Clearly, $\tilde{a}^0(\cdot)$ is a solution of Eq.\ \eqref{eq-ef1} with the initial datum $a_0$.
Due to Lemma \ref{rnls-lem-apr},  $a^0(\cdot)$ is a mild solution of Eq.\ \eqref{eq-ef1} with $a^0(0)=a_0$.  Hence we have $a^0(\tau)=\tilde{a}^0(\tau),$ $\tau\in[0,1]$, particularly, $a^0(\tau_0)=\tilde{a}^0(\tau_0)$. This contradicts with \eqref{contradict},
so the convergence \eqref{convergence1} is uniform in $v_0\in B(M_0, h^{s_*})$.

Since 
$$
|I(a)-I(\tilde a)|^\sim_{s_1} \le 4 |a-\tilde a|_{s_1} (|a|_{s_1}  +| \tilde a|_{s_1} ),
$$
then the convergence \eqref{convergence1} implies the statement of Theorem \ref{m-theorem}.

\section{The randomly forced case}\label{sec-random}

We study here the effect of the addition  a random forcing to
Eq.~\eqref{m-eq}. Namely, we consider equation \eqref{y1}. We suppose that 
$$
B_s=2\sum_{j=1}^\infty \lambda_j^{2s} b_j^2<\infty\quad \text{
for $s=s_*\in(d/2+1, n]$,}
$$
and impose  a restriction on the nonlinearity  $\cP$ by assuming 
that there exists $\bar N\in\N$ and 
for each $s\in(d/2+1,n]$ there exists $C_s$ such that 
\begin{equation}\label{restr-P}
\left\|\cP(\nabla u,u)\right\|_{s-1}  \le C_s 
(1+\|u\|_s)^{\bar N}\,, \quad \forall\, u\in H^s \,
\end{equation}
(this assumption holds e.g. 
 if $\cP(\nabla u,u)$ is a polynomial in $(u, \nabla u)$).

Passing to the slow time $\tau=\epsilon t$, Eq.~\eqref{y1} becomes
(cf. \eqref{m-eq2})
 \be\label{y12}
 \dot{u}+\epsilon^{-1}i(-\Delta+V(x))u=\mu\Delta u+\mathcal{P}(\nabla
 u, u)+\frac{d}{d\tau}\sum_{k=1}^{\infty}b_k \bb_ke_k(x),\quad
u=u(\tau,x),
 \ee
which, in the $v$-variables, takes the form (cf. \eqref{eq-v1})
\begin{equation}\label{feq-v1}
  dv_k+\epsilon^{-1}i\lambda_kv_k\,d\tau
  =\left(-\mu\lambda_kv_k+P_k(v)\right) d\tau +\sum_{l=1}^\infty
  \Psi_{kl}b_ld\bb_l \ ,\quad
  k\in\mathbb{N}\ ,
\end{equation}
where we have denoted by $\{\Psi_{kl},k,l\ge 1\}$ the matrix of the operator  $\Psi$
(see \eqref{fourier}) with respect to the basis $\{e_k\}$ in $H^0$ and
$\{\zeta_k\}$ in $h^0$.
We assume

\noindent\textbf{Assumption A\textprime}. \textit{There exist $s_*\in
  (d/2+1,n]$ and an  $\epsilon$-independent $T>0$ such that for 
 any $u_0\in H^{s_*}$, Eq.~\eqref{y12} has a unique strong
solution $u(\tau,x)$, $0\le\tau\le T$, 
equal to $u_0$ at $\tau=0$. Furthermore, 
for each $p$ there exists a  $C=
C_p(\left\|u_0\right\|_{s_*},B_{s_*},T)$ such that}  
\begin{equation}\label{ass-a2}
\E\sup_{0 \le \tau \le T} \left
\|u(\tau) \right\|^p_{s_*}\le C\,. 
\end{equation}

\begin{remark}\label{r1}
The Assumption A\textprime \  is not too restrictive. In particular, in \cite{Kuk11} it is verified for
equations \eqref{y1} if $\mu>0$ and ${\mathcal P}(u) =
-u+zf_p(|u|^2)u$, where $f_p(r)$ is 
a smooth function, equal $|r|^p$ for $|r|\ge1$, and $\Im z\le0, \Re
z\le0$. The degree $p$ is any real 
number if $d=1,2$ and $p<2/(d-2)$ if $d\ge3$.
\end{remark}

Under this assumption, a result analogous to
Theorem~\ref{m-theorem} holds. Namely, the limiting behaviour of the
action variables $I_k$ (see \eqref{action-angle}) is described by the
stochastically forced effective equation (cf. \eqref{eq-ef1})
  \begin{equation}\label{feq-ef1}
  d\tilde a_k=\left(-\mu\lambda_k\tilde a_k+R_k(\tilde
  a)\right)d\tau+\sum_{l=1}^\infty B_{kl}d\bb_l\ ,\quad 
  k\in\mathbb{N}\ ,
  \end{equation}
where we have defined $\{B_{kr}, k,r\ge 1\}$ as the principal square root
  of the real  matrix 
\begin{equation}\label{diff-ef1}
A_{kr}=\left\{\begin{array}{cc}
\sum_lb_l^2 \Psi_{kl} \Psi_{rl} &\mbox{if }\lambda_k=\lambda_r\,,\\
0 &\mbox{else}\,,
\end{array}
\right. \ .
\end{equation}
which defines a nonnegative selfadjoint compact operator in $h^0$.
Note that since $R$ is locally Lipschitz by Lemma~\ref{lem-r-lip}, then 
strong solutions for \eqref{feq-ef1} exist and are unique till the
stopping time $\tau_K = \inf\{\tau\ge0: |\tilde a(\tau)|_{s_*}=K\}$, where
$K$ is any positive number.

 In the theorem below $v^\epsilon(\tau)$ denotes a solution of \eqref{feq-v1}
with the initial value $v_0\in h^{s_*}$.

  \begin{theorem}\label{f-theorem}
  If Assumption A\textprime{}  holds, there exists a unique strong 
  solution $\tilde a(\tau)$, $0\le\tau\le T$,  of equation
  \eqref{feq-ef1} such that $\tilde a(0)=v_0=\Psi(u_0)\in h^{s_*}$,  and 
$$
\cD\left(I(v^\epsilon(\tau))\right)\strela \cD\left(I(\tilde
a(\tau))\right)  \quad\mbox{as }\epsilon \to 0\ ,
$$
 in $C([0,T],h^{s_1}_I)$, for any $s_1<s_*$. 
\end{theorem}

In the theorem's assertion and below the arrow 
 $\strela$ stands for the weak convergence of measures. 
Let us assume further:

\noindent\textbf{Assumption B\textprime{}.} {\it 
 i)
Eq.~\eqref{y1} has a unique strong solution $u(\tau), u(0)=u_0\in H^{s_*}$,
 defined for $\tau\ge0$, and 
\begin{equation}\label{ass-a3}
\E\sup_{\theta \le \tau \le \theta+1} 
\|u(\tau)  \|^p_{s_*}\le C\quad \text{for any $\theta\ge0$}\,,
\end{equation}
where $C= C(\left\|u_0\right\|_{s_*},B_{s_*})$.

ii) Eq.~\eqref{feq-ef1} has a unique stationary measure $\mu^0$ and is mixing. }

\begin{remark}\label{r2}
The assumption i) is fulfilled, for example, for equations, discussed in Remark~\ref{r1}. 
Assumption~ii) holds trivially if
 for a.e. realisation of the random force any two solutions of 
Eq. \eqref{feq-ef1} 
 converge exponentially fast.\footnote{This is  fulfilled,  for example, if i)  holds and 
$\cP(u)=-u+\cP_0(u)$, where the Lipschitz constant of $\cP_0$ is less than one.
}
For less trivial examples, corresponding to
perturbations of linear systems with non-resonant or completely resonant spectra, see \cite{Kuk11, skam2013}. 
\end{remark}
  
  Assumption  B\textprime{} i) and the Bogolyubov-Krylov argument, applies for solutions, starting 
  from 0,  imply that Eq.~\eqref{y1} has a stationary 
  measure $\mu^\epsilon$, supported by the space $H^{s_*}$, and inheriting estimates 
  \eqref{ass-a3}. 
  
 \begin{theorem}\label{t.stat}
Let us suppose that Assumptions A\textprime{} and B\textprime{}  hold. Then 
 \be\label{conv1}
\lim_{\epsilon \to 0}  \mu^{\epsilon} = \mu^0\ ,
\ee
weakly in $h^{s_1}$, for any $s_1<s_*$. The measure $\mu^0$ is invariant with respect to
transformations $\Phi_{t\Lambda}$, $t\in\R$. 
If, in addition, \eqref{y1} is
mixing and $\mu^\epsilon$ is its unique stationary measure, then for  any
 solution $u^\epsilon(t)$ of \eqref{y1} with $\epsilon$-independent
 initial data $u_0\in H^{s_*}$, we have 
$$
\lim_{\epsilon\to 0}\lim_{t\to \infty}
\cD\left( v^\epsilon(t)
\right) =  \mu^0\,,
$$
where $v^\epsilon(t) =  \Psi\left(u^\epsilon(t)\right)$. 
\end{theorem}

For examples of mixing equations \eqref{y1}  see \cite{Kuk11} and
references in that work.  
In particular,   \eqref{y1} is mixing if $\cP(u,\nabla u)=\cP(u)$ is a 
 smooth function such that all its derivatives are bounded uniformly in $u$, 
 cf. Remark~\ref{r2}. 

For the case when  the spectrum
$\Lambda$ is  non-resonant (see Example~\ref{ex1}) or is completely resonant, i.e. \eqref{complres} holds, 
the theorem was proved in  \cite{Kuk11, skam2013}.

The proofs of  Theorem \ref{f-theorem} and \ref{t.stat}
 closely follow the arguments in~\cite{Kuk11,skam2013,bplane}. Proof of Theorem~\ref{t.stat}, in
 addition, uses some technical ideas from \cite{Dym} (see there Corollary~4.2). 
 The proofs 
  are given, respectively, in Section~\ref{sec-r1} and Section~\ref{sec-r2}.

\section{Proof of Theorem~\ref{f-theorem}} \label{sec-r1} 
As in the proof of Theorem~\ref{m-theorem}, let us assume, without
loss of generality, that $T=1$, 
 $s_1>d/2+1$ and $s_1> s_*-2$ (recall that
$s_1<s_*$ and $s_*\in (d/2+1,n]$). 

Following the suite  of \cite{skam2013} (see also
\cite{bplane}) we pass once again to the $a$-variables, defined in
\eqref{def-a}). In view of \eqref{feq-v1}, they
satisfy the system (cf. \eqref{eq-a1})
\begin{equation}\label{feq-a1}
  da_k=\left(-
  \mu\lambda_ka_k+Y_k(a,\epsilon^{-1} \tau) \right)d\tau
  + e^{i\epsilon^{-1}\lambda_k\tau} \sum_l \Psi_{kl}b_ld\bb_l
  ,\quad k\in\mathbb{N}\ ,
\end{equation}
where $Y$  is defined in \eqref{def-Y}. For any $p$ we denote
$$
X^p=C([0,1],h^p)\,, \qquad X^p_I=C([0,1],h^p_I)\,. 
$$

Let $a^\epsilon$ be a solution of \eqref{feq-a1} such that
$a^\epsilon(0)=v_0=\Psi(u_0)\in h^{s_*}$; we will often write $a$ for $a^\epsilon$
to shorten notation. Denote the white noise in \eqref{feq-a1} as $\dot
\zeta(t,x)$ and denote $U_1(\tau)= Y(a(\tau),\epsilon^{-1} \tau)$,
$U_2(\tau)=-\hA a(\tau)$. 
Then 
$$
\dot a-\dot \zeta= U_1+U_2\,.
$$
 In view of \eqref{restr-P},
$\left\|U_1\right\|_{s_*-1}=|P(v)|_{s_*-1}\le 
C(1+ \|u(\tau)\|_{s_*}^{\bar N})
$. So, by
\eqref{ass-a2},   
$$
\E\int_\tau^{(\tau+\tau')\wedge     1}\left\|U_1\right\|_{s_*-1}\,dt\le
C\int_\tau^{(\tau+\tau')\wedge     1}\E C(1+ \|u(t)\|_{s_*}^{\bar N})
 \,dt \le
C(\left\|u_0\right\|_{s_*},B_{s_*}) \tau'\ , 
$$
for any $\tau\in[0,1]$ and $\tau'>0$. 
Similar,
$$
\E \int_\tau^{(\tau+\tau')\wedge     1 } \left\|U_2\right\|_{s_*-2}\,dt
\le 
\mu C \E\int_\tau^{(\tau+\tau')\wedge     1} 
\left\|u\right\|_{s_*}\le \mu C(\left\|u_0\right\|_{s_*},B_{s_*}) \tau'\ . 
$$
Hence,  there exists $\gamma>0$ such that
$$
\E\left\|(a-\zeta)((\tau+\tau')\wedge
1)-(a-\zeta)(\tau)\right\|_{s_1}\le C(\left\|u_0\right\|_{s_*},B_{s_*})
{\tau'}^{\gamma}\ ,
$$
in virtue of the interpolation  and H\"older
inequalities (cf. \eqref{new2}).
It is classical that 
$$
\PP\{ \|\zeta\|_{C^{1/3}([0,1],h^{s_1})}\le R_3\}\to1\quad\text{as}\quad R_3\to\infty\,.
$$
In view of what was said, for any $\delta>0$ there is a  set
$Q^1_\delta\subset X^{s_1}$,
formed by equicontinuous functions,  
such that 
$$
\PP\{a^\epsilon\in Q^1_\delta\} \ge 1-\delta\,,
$$
for each $\epsilon$.
By \eqref{ass-a2},
$$
\PP\{\|a^\epsilon\|_{X^{s_*}} \ge C\delta^{-1}\} \le \delta\,,
$$
 for a suitable $C$, uniformly in $\epsilon$.
Consider the set 
$$
Q_\delta =\left\{ a^\epsilon\in Q^1_\delta: \left\|a\right\|_{X^{s_*}} \le
C\delta^{-1}\right\}\ .
$$
Then $\PP\{a^\epsilon\in Q_\delta\} \ge 1-2\delta$, for each $\epsilon$. 
By this relation and the Arzel\`a-Ascoli theorem (e.g., see \cite{KelNam}, \S8),
the set of laws 
$\{\cD(a^\epsilon(\cdot)),\ 0<\epsilon \le1\}$, 
is tight in $X^{s_1}$. So  by the Prokhorov theorem 
 there is a sequence $\epsilon_l\to 0$ 
and a Borel measure $\cQ^0$ on $X^{s_1}$ such that
\be\label{z5}
\cD(\ai^{\epsilon_l}(\cdot)) \strela \cQ^0\quad\text{as}\quad
\epsilon_l\to 0\,.
\ee
 Accordingly, due to \eqref{newesti}, for 
actions of solutions $v^\epsilon$ we have the convergence 
\be\label{z55}
\cD\left( I\left( v^{\epsilon_l}(\cdot)\right)\right) \strela I\circ
\cQ^0\quad\text{as}\quad \epsilon_l\to 0\,,
\ee 
 in  $X_I^{s_1}
$.

Theorem~\ref{f-theorem} follows then as a simple corollary from
\begin{proposition}\label{prop1}
There exists a unique weak solution $ a(\tau)$ of the effective
equation \eqref{feq-ef1} such that $\cD(a)=\cQ^0$, $a(0)=v^0$ 
 a.s.; and the
convergences \eqref{z5} and \eqref{z55} hold as $\epsilon \to 0$.
\end{proposition}

\begin{proof}
The  proof follows the Khasminski scheme (see \cite{Khas68, FW03}), as expounded
in \cite{skam2013}. Namely, we show that the limiting measure $\cQ^0$ is a martingale 
solution of the limiting equation, which turns out to be exactly the equation \eqref{feq-ef1}. 
Since the equation has a unique solution, then the convergences \eqref{z5}, \eqref{z55}
hold as $\epsilon\to0$. 

For $\tau\in[0,1]$ consider the processes 
$$
N^{\epsilon_l}_k=\ai^{\epsilon_l}_k(\tau) - \int_0^\tau \left(-\mu
\lambda_k \ai^{\epsilon_l}_k(s)+R_k(\ai^{\epsilon_l}(s))\right)\,d s\ ,
\quad k \ge 1\ 
$$
(cf. Eq. \eqref{feq-ef1}). 
Due to \eqref{feq-a1}  we  write $N^{\epsilon_l}_k$ as
$$
N^{\epsilon_l}_k(\tau)= \widetilde N^{\epsilon_l}_k(\tau)+\overline
N^{\epsilon_l}_k(\tau)\ ,
$$
where $\widetilde N^{\epsilon_l}_k(\tau)=a^{\epsilon_l}(\tau)-
\int_0^\tau (-\mu\lambda_k
\ai^{\epsilon_l}(s)+Y_k(\ai^{\epsilon_l}(s),\epsilon_l^{-1}s)) 
ds $ is a $\cQ^0$ martingale and the disparity $\overline N^{\epsilon_l}_k$ is 
$$ 
\overline
N^{\epsilon_l}_k(\tau)=\int_0^{\tau }
\cY_k(\ai^{\epsilon_l}(s),\epsilon_l^{-1}s)
 ds \,
$$
(as before, $\cY(a,t) = Y(a,t)- R(a)$). 

The key point is then a stochastic counterpart of
Lemma~\ref{rnls-lem-apr}, which is proved below:
\begin{lemma}\label{lem-rot}
For every $k\in \N$,  $\E\,\mathfrak{A}^\epsilon_k\to0$ as $\epsilon\to0$, where 
$$
\mathfrak{A}^\epsilon_k = \max_{0\le \tilde \tau \le
  1}\left|\int_{0}^{\tilde{\tau}}\mathcal{Y}_k(a^\epsilon(\tau),
\epsilon^{-1}\tau)d\tau\right|  \,.
$$
\end{lemma}

This lemma and the convergence \eqref{z5} 
imply that the processes
$$
N_k(\tau)= \ai_k(\tau)-\int_0^\tau\left(-\mu\lambda_k\ai_k+ R_k(\ai)\right)\,ds\ , \quad k\ge 1\ ,
$$
are $\cQ^0$ martingales, considered on the probability space $(\Omega=X^{s_1}, \cF, Q^0)$ 
($\cF$ is the Borel sigma-algebra), given the natural filtration $(\cF_\tau, 0\le\tau\le1)$. For 
 details see \cite{kuksinjpam}, Proposition 6.3). 

Consider then the diffusion matrix $\{\cA_{kr},k,r\ge 1\}$ for the system
\eqref{feq-a1}, i.e., 
$$
\cA_{kr}= \exp(i\epsilon^{-1}\tau(\lambda_k-\lambda_r))\sum_{l=1}^\infty
b_l^2\Psi_{kl}\bar \Psi_{rl} \ .
$$
Clearly, $\int_0^{\tilde \tau}\cA_{kr} d\tau\to A_{kr}\tilde \tau$, as $\epsilon\to
0$, where $A$ denotes the diffusion matrix for the system
\eqref{feq-ef1} (cf. \eqref{diff-ef1}). Similar to
Lemma~\ref{lem-rot}, we also find  that 
$$
\E \max_{0\le \tilde \tau\le  1} \left|\int_0^{\tilde \tau }
\cY_k(\ai^\epsilon(\tau),\epsilon^{-1}\tau) d\tau  \right|^2 \to 0\quad \mbox{as }
\epsilon\to 0\ . 
$$
Then, using the same argument as before, we see that the processes
\begin{equation*}
\begin{split}
N_{k} (\tau) N_{r}(\tau)- A_{kr}\tau =& \left(\widetilde
N_k\widetilde N_r -\int_0^\tau \cA_{kr} ds\right)\\
& +\left(\overline
N_k\overline N_r+\overline N_k \widetilde N_r+\widetilde N_k \overline
N_l - \int_0^\tau (\cA_{kr}-A_{kr})\,ds\right)
\end{split}
\end{equation*}
 are $\cQ^0$
martingales. That is, $\cQ^0$ is a solution of the martingale
problem with the  drift $R$ and the diffusion $A$ (see \cite{SV}), so 
 the assertion follows. 
\end{proof}

\begin{proof}[Proof of Lemma~\ref{lem-rot}]
We adopt a convenient notation from our previous publications. Namely, we denote by $\vk(r)$
various functions of $r$ such that $\vk\to0$  
as $r \to\infty$. 
 We write $\vk(r;M)$ to indicate that $\vk(r)$ 
depends on a parameter $M$. Besides for events $Q$ and $O$ and  a
random variable  $f$ we write $\PP_O(Q)=\PP(O\cap Q)$ and 
$\E_O(f)=\E(\chi_O\, f)$. 

The constants below may depend on $k$, but
this dependence is not indicated since $k$ is fixed through the proof. By $M\ge1$ we denote a 
constant which will be specified later.
Denote by $\Omega_M=\Omega^\epsilon_M$ the event
$$
\Omega_M=\left\{\sup_{0\le\tau\le 1} \left|\ai^\epsilon(\tau)\right|_{s_*}\le M\right\}\ .
$$
Then, by \eqref{ass-a2},
\be\label{06}
\PP(\Omega_M^c)\le \vk(M).
\ee

In view of  Lemma \ref{lem-r-lip}  (ii) and \eqref{restr-P}, for any $t\in[0, \epsilon^{-1}]$ and any
$a\in h^{s_*}$ the difference  $\cY=Y-R$  satisfies
\begin{equation}\label{kva}
\left|\cY_k(a,t)\right| \le \left|Y_k(a,t)\right|+
\left|R_k(a)\right|\le
\left|P_k(v)\right|+ \left|R_k(a)\right|\le
C (1+  \left|a\right|_{s_*})^{\bar N} \,.
\end{equation}
Using this and 
 \eqref{06}  we get  
\begin{equation}\label{6.88}
\begin{split}
\E_{\Omega_M^c}
\mathfrak A^\epsilon_{k}&\le \int_0^1
\E_{\Omega_M^c}|\cY_k(a(\tau),\epsilon^{-1} \tau) | d\tau\\
&\le C
\left(\IP(\Omega_M^c)\right)^{1/2}  
\int_0^1\left(\E (1+  \left|a\right|_{s_*})^{2\bar N}
\right)^{1/2} d\tau \le
\vk(M)\ . 
\end{split}
\end{equation}

To estimate $\E_{\Omega_M} \mathfrak A^\epsilon_{k}$, as in Lemma \ref{rnls-lem-apr} 
 we 
consider a partition of $[0,1]$  by the points 
$$
 b_n= nL,\quad 0\le n\le N-1, \; \ b_{N-1}\ge 1-L, \ b_N=1\,,
  \qquad L=\epsilon^{1/2}\,,
$$
$N\sim 1/ L$. 
 Let us denote
$$
\eta_l= \int_{ b_l}^{  b_{l+1}}  \cY_k(a
(\tau),\epsilon^{-1}\tau)  d\tau\ ,\quad 
0\le l\le N-1\,.
$$
Since for $\omega\in\Omega_M$  and any  $\tau'<\tau''$ such that
$\tau''-\tau'\le L$,  in view of \eqref{kva} we have $
\left|\int_{\tau'}^{\tau''} \cY_k(a(\tau),\epsilon^{-1} \tau)
d\tau\right| \le L C(M)$, then
\begin{equation}\label{5.170}
\E_{\Omega_M}
\mathfrak A^\epsilon_{k} \le LC(M)+\E_{\Omega_M}\sum_{l=0}^{N-1}|\eta_l|\,.
\end{equation}

Let us fix  any $\bar s>d/2+1$, $s_*-2<\bar s<s_*$,  sufficiently
small $\gamma>0$, and consider the event  
$$
\cF_l=\left\{\sup_{ b_l\le \tau\le b_{l+1}}\left|a^\epsilon(\tau)
-a^\epsilon( b_l)\right|_{\bar s} \ge L^\gamma\right\}\ .
$$
By the equicontinuity of the processes $\{a^\epsilon(\tau)\}$ on suitable events 
with arbitrarily close to one  $\epsilon$-independent probability 
 (as shown above), the probability of  $\IP(\cF_l) $  goes to zero
with $L$, uniformly in $l$ and $\epsilon$. Since $|\eta_l|\le C(M) L$ for $\omega\in\Omega_M$ and 
 each $l$, then 
\begin{equation}\label{5.211}
\sum_{l=0}^{N-1}\left|\E_{\Omega_M} |\eta_l|-
\E_{\Omega_M\backslash \cF_l}|\eta_l|  \right| \le
{C(M)}{L}\sum_{l=0}^{N-1}\IP_{\Omega_M}(\cF_l)\le C(M) \vk(L^{-1})\ , 
\end{equation}
and it remains to estimate $\sum_l \E_{\Omega_M\backslash \cF_l} |\eta_l|$. 

We have
\begin{equation*}
 \begin{split}
\  |\eta_l|   &\le \left|
\int_{ b_l}^{ b_{l+1}} \left(  \cY_k(a(\tau),\epsilon^{-1} \tau)-
\cY_k( a(  b_l),\epsilon^{-1} \tau)\right) d\tau\right|\\
&+ \left|
\int_{  b_l}^{ b_{l+1}}\left(  \cY_k(a( b_l),\epsilon^{-1} \tau)\right)
d\tau \right| 
 =:\Upsilon^1_l+\Upsilon^2_l\ .
\end{split}
\end{equation*}
By \eqref{lip-p-v} and Lemma~\ref{lem-r-lip} (ii), in
 $\Omega_M$ the following inequality hold:
\begin{equation*}
\left|  \cY_k(a(\tau),\epsilon^{-1} \tau)- \cY_k(a(
 b_l),\epsilon^{-1} \tau) \right|\le C(M) \left|a(\tau)- a(
 b_l)\right|_{\bar s}\ .
\end{equation*}
So that, by the definition of $\cF_l$, 
\begin{equation}\label{5.002}
  \sum_l  
\E_{\Omega_M\backslash \cF_l}  
 \Upsilon^1_l  \le L^\gamma C(M)=\vk( \epsilon^{-1};M) \ . 
\end{equation}

It remains to estimate the expectation of $\sum\Upsilon^2_l$.
 In view of \eqref{foralberto} (with $M_1=M$) 
 we have 
\begin{equation}\label{5.003}
 \sum_l  \E_{ \Omega_M\backslash\cF_l}
\Upsilon^2_l  \le  NL\vk_1(\epsilon^{-1};M ) =
\vk(\epsilon^{-1};M ). 
\end{equation}

Now the inequalities \eqref{6.88}--\eqref{5.003} jointly imply that 
\begin{equation*}
\begin{split}
\E\,
\mathfrak A^\epsilon_k \le \,& \vk(M)+ \vk(\epsilon^{-1};M)\ .  
\end{split}
\end{equation*}
Choosing first $M$ large  and then $\epsilon$ small
  we make the r.h.s.   arbitrarily
small. This proves the lemma. 
\end{proof}

Lemma \ref{lem-rot} estimates integrals of the differences
$$
e^{i\epsilon^{-1} \tau \lambda_k} P_k\big(\Phi_{-\epsilon^{-1} \tau \lambda_k} (a^\epsilon(\tau)\big) -\lan P\ran_{\Lambda,k}
(a^\epsilon(\tau) )\,.
$$
Similar result holds if we replace the averaging $\lan \cdot\ran_{\Lambda,k}$ by  $\lan\lan \cdot\ran\ran_{\Lambda}$
and the function $P_k$ by any Lipschitz function: 

\begin{lemma}\label{l62}
Let $f\in Lip_1(h^{s_1})=: Lip_1$ (i.e., $f$ is a bounded Lipschitz function on $h^{s_1}$). Then 

i) 
$\;\;
\E \int_0^1 \big( f(\Phi_{-\tau\epsilon^{-1} \Lambda}a^\epsilon(\tau)) - 
\fla (a^\epsilon(\tau)) \big)\, d\tau \to0\ $ as $\ \epsilon\to0\,;
$

ii) if  in i) $f$ is replaced by 
 $ f^\theta=f\circ\Phi_\theta$, $\theta\in\T^\infty$, then the rate of convergence  does not depend on $\theta$.
\end{lemma}
\begin{proof}
To get i) we literally  repeat the proof of Lemma \ref{lem-rot},  using Lemma~\ref{lnew} instead of Lemma~\ref{lem-def-av}. 
The assertion ii) follows from  Lemma~\ref{lnew} and item (iv) of Lemma~\ref{lem-def-av}. 
\end{proof}

\section{Proof of Theorem~\ref{t.stat}}\label{sec-r2}

Let  $v^\epsilon(\tau),\, 0\le \tau \le1 $, be a stationary solution for Eq.~\eqref{y1} such that 
$\cD( v^\epsilon(\tau))\equiv \mu^\epsilon$, 
and let $a^\epsilon (\tau) = \Phi_{\epsilon^{-1} \Lambda\tau} v^\epsilon(\tau)$ be its interaction representation.  Since
$v$ inherits the a-priori estimate \eqref{ass-a3} (with $u_0=0$), then an analogy of the convergence \eqref{z5}
holds for a suitable sequence $\epsilon_l\to0$. 
The argument from the proof of Proposition~\ref{prop1} applies and imply that 
\be\label{9.1}
\begin{split}
\cD(a^{\epsilon_l}(\cdot)) \strela \cD (a^0(\cdot))\quad \text{in}\quad X^{s_1}\quad\text{as}\quad
\epsilon_l\to0\,,
\end{split}
\ee
where $a^0$ is a weak solution of \eqref{feq-ef1}. We may also assume that 
\be\label{9.11}
\begin{split}
\mu^{\epsilon_l} \strela \bar \mu^0\quad \text{in}\quad h^{s_1}\,,
\end{split}
\ee
for some measure $\bar \mu^0$.

Let us take any $f\in Lip_1(h^{s_1})$. Then 
$$
\E \int_0^1 f(v^\epsilon (\tau))\,d\tau =\E \int_0^1 f(\Phi_{-\epsilon^{-1}\Lambda\tau}a^\epsilon(\tau))
\,d\tau\,.
$$
Applying to the second integral Lemma \ref{l62} we find that 
\be\label{9.2}
 \int_0^1 \E f(v^\epsilon (\tau))\,d\tau = \int_0^1 \E \fla
(a^\epsilon(\tau)) \,d\tau + \varkappa(\epsilon^{-1})\,.
\ee
Since the function $\fla$ is invariant with respect to transformations $\Phi_{\Lambda t}$, $t\in\R$
(see item b) of Lemma~\ref{lnew}), 
then $\fla (a^\epsilon(\tau)) = \fla(v^\epsilon(\tau))$. So both integrands in \eqref{9.2} are independent from 
$\tau$, and 
\be\label{9.3}
  \E f(v^\epsilon (\tau))=  \E \fla
(a^\epsilon(\tau))  + \varkappa(\epsilon^{-1})\quad \forall\,\tau\,.
\ee
Now let us take for $f$ the function
$\tilde f = \tilde f_{\epsilon^{-1}\tau}
=f\circ\Phi_{\epsilon^{-1} \Lambda\tau}$ (which also belongs to $Lip_1(h^{s_1})$). Then 
$$
\E f(a^\epsilon(\tau)) =\E \tilde f(v^\epsilon(\tau)) =\E \lan\lan \tilde f\ran\ran_\Lambda(a^\epsilon(\tau))+ \varkappa(\epsilon^{-1})
=\E\fla (a^\epsilon(\tau)) +
 \varkappa(\epsilon^{-1})\,,
$$
where $\varkappa$ may be chosen the same for all functions $\tilde f$ in 
view of Lemma~\ref{l62}\,i).
Comparing this with \eqref{9.3} and using \eqref{9.11}  we find that 
$$
\E f(a^{\epsilon_l}(\tau))\strela \lan f,\bar \mu^0\ran\quad\text{as}\quad \epsilon_l\to0, 
$$
for each $\tau$. Therefore, in virtue of \eqref{9.1},
$\cD(a^0(\tau))\equiv\bar \mu^0$. So $a^0(\tau)$ is a stationary solution 
for \eqref{feq-ef1}, and $\bar \mu^0$ is a stationary measure for this
equation. Since the latter is unique, $\bar \mu^0\equiv \mu^0$, and
\eqref{9.11} implies the convergence \eqref{conv1}.

Replacing in \eqref{9.3} $f$ by $\tilde f_t$ and using Lemma~\ref{lnew}\,b) we see that 
$$
\lan f,\Phi_{\Lambda t}\circ\mu^\epsilon\ran = \lan \tilde f_t,\mu^\epsilon\ran = \lan f,\mu^\epsilon\ran+\varkappa(\epsilon^{-1}). 
$$
Passing to the limit as $\epsilon\to0$ we get the claimed invariance of the measure $\mu^0$. Finally, the last 
  assertion  immediately follows from \eqref{conv1}.
  $\qed$
\bigskip

\appendix
\section{}
Consider the  CGL equation  \eqref{dcdse}, 
where $\mathcal{P}: \mathbb{C}^{d(d+1)/2+d+1}\times \TTT\to\mathbb{C}$ is a $C^{\infty}$-smooth function. We write it in the $v$-variables and 
 slow time $\tau=\epsilon t$:
\[
\dot{v}_k+\epsilon^{-1}i\lambda_kv_k=P_k(v),\quad k\in\mathbb{N},
\]
where 
\[P(v):=(P_k(v),k\in\mathbb{N})=\Psi(\mathcal{P}(\nabla^2u,\nabla u, u,x)), \; \;\; u=\Psi^{-1}v,\]
and  introduce the effective equation 
\begin{equation}\label{appendix-ef}
\dot{\tilde{a}}=\langle P\rangle_{\Lambda}(\tilde{a}).
\end{equation}
By Lemma \ref{lem-smooth}  $P$ defines smooth locally Lipschitz 
mappings $h^s\to h^{s-2}$ for $s>2+d/2$. So by a version 
of Lemma~\ref{lem-r-lip}, $\langle P\rangle_{\Lambda}\in Lip_{loc}(h^s; h^{s-2})$
for $s>2+d/2$.
Assume that

\noindent{\bf Assumption E}: {\it There exists $s_0\in(d/2,n]$ such that the effective equation 
\eqref{appendix-ef} is locally  well posed in the Hilbert spaces $h^s$, with $s\in [s_0, n]\cap\N$.}
\medskip

Let $ u^{\epsilon}(t,x)$ be a solution of Eq.~\eqref{dcdse} with initial datum $u_0\in H^s$, $v^{\epsilon}(\tau)=\Psi(u(\epsilon^{-1}\tau, x))$, and $\tilde{a}(\tau)$ be a solution of Eq.\ \eqref{appendix-ef}  with initial datum $\Psi(u_0)$. Then we have the following result:
\begin{theorem}\label{thm-appendix}
If Assumptions A and E hold and $s>\max\{s_0+2, d/2+4\}$, then the solution of the 
effective equation exists for $0\leqslant \tau\leqslant T$, and for any $s_1<s$ we have
$$
I(v^{\epsilon}(\cdot))
\xrightarrow[\epsilon\to0] {}
I(\tilde{a}(\cdot))\quad \text{in}\quad C([0,T], h_I^{s_1}).\label{converge-A}
$$
\end{theorem}

The proof of this theorem  follows  that of Theorem \ref{m-theorem}, with slight modifications. 
Cf. \cite{hgdcds}, where the result is proven for the non-resonant case. 

\bibliography{WNPDES}

\end{document}